\documentclass[12pt]{amsart}
\usepackage[utf8]{inputenc}
\usepackage[english]{babel}
\usepackage{amsmath}
\usepackage{amssymb}
\usepackage{amsfonts}
\usepackage{stmaryrd} 

\usepackage{calrsfs} 

\usepackage{longtable}

\usepackage[left=2.5cm,right=2.5cm,top=2.5cm,bottom=2.5cm,bindingoffset=0cm]{geometry}

\newtheorem{thm}{Theorem}

\newtheorem{lem}{Lemma}[section]
\newtheorem{prop}[lem]{Proposition}
\newtheorem{cor}[thm]{Corollary}

\theoremstyle{definition}
\newtheorem*{prob*}{Problem}
\newtheorem{ex}{Example}
\newcommand{\fX}{\mathfrak{X}}
\newcommand{\mX}{\operatorname{m}_{\mathfrak{X}}}
\newcommand{\smX}{\operatorname{sm}_{\mathfrak{X}}}

\begin{document}

\title[On embedding theorems for $\fX$-subgroups]{On embedding theorems for $\fX$-subgroups}

\author{Wenbin Guo, Danila O. Revin, and Andrey V. Vasil'ev}

\address{Wenbin Guo
\newline\hphantom{iii} School of Science, Hainan University, Haikou, Hainan, 570228, P.R. China}
\email{wbguo@ustc.edu.cn}

\address{Danila O. Revin
\newline\hphantom{iii} Sobolev Institute of Mathematics, Novosibirsk 630090, Russia}
\email{revin@math.nsc.ru}

\address{Andrey V. Vasil'ev
\newline\hphantom{iii} Sobolev Institute of Mathematics, Novosibirsk 630090, Russia
\newline\hphantom{iii} School of Science, Hainan University, Haikou, Hainan, 570228, P.R. China}
\email{vasand@math.nsc.ru}

\thanks{The first and third author were supported by the Natural Science Foundation of China (No. 12171126).}

\begin{abstract}
Let $\fX$ be a class of finite groups closed under subgroups, homomorphic images, and extensions. We study the question which goes back to the lectures of H. Wielandt in 1963-64: For a given $\fX$-subgroup $K$ and maximal $\fX$-subgroup $H$, is it possible to see embeddability of $K$ in $H$ (up to conjugacy) by their projections onto the factors of a fixed subnormal series. On the one hand, we construct examples where $K$ has the same projections as some subgroup of $H$ but is not conjugate to any subgroup of~$H$. On the other hand, we prove that if $K$ normalizes the projections of a subgroup $H$, then $K$ is conjugate to a subgroup of $H$ even in the more general case when $H$ is a submaximal $\mathfrak{X}$-subgroup.
\end{abstract}

\maketitle

\section{Main concepts and results}\label{sec1}

Let $G$ be a finite group. We fix a series
\begin{equation*}
 \label{ser}
 G=G_0\geqslant G_1\geqslant \dots\geqslant G_n=1.\eqno{(*)}
\end{equation*}
As usual, we say that the series $(*)$ is {\em normal} or  {\em subnormal}, if, respectively, $G\trianglerighteqslant G_i$ or $G_{i-1}\trianglerighteqslant G_i$ for each $i=1,\dots, n$. Given a subgroup $H$ of $G$, put $$H_i=H\cap G_i\ \text{and}\ H^i= H_{i-1}G_i/G_i.$$
In particular, $G^i=G_{i-1}/G_i$ is $i$th section or $i$th factor of the series~$(*)$. The subgroup $H^i$ of $G^i$ is said to be {\em the projection of $H$ onto the factor~$G^i$}.

In general, two subgroups having the same projections on the factors of~$(*)$ need not even be isomorphic. The situation turns out to be essentially different if we consider not arbitrary subgroups, but maximal subgroups in some complete class. Recall that, according to Wielandt \cite{Wie4,Wie3}, a nonempty class of groups $\fX$ is called {\em complete} if it is closed under taking subgroups, homomorphic images, and extensions. A maximal $\fX$-subgroup $H$ of $G$ will also be called {\em $\fX$-maximal} and written $H\in\mX(G)$. Wielandt proved \cite[Theorem~14.1(3)]{Wie4} that \textit{if $H,K\in\mX(G)$ and the projections of these subgroups onto the sections of the normal series~$(*)$ coincide, then $H$ and $K$ are conjugate in~$G$}. In other words, a maximal $\fX$-subgroup is uniquely determined by its projections onto the factors of the normal series up to conjugacy.

A similar assertion for a subnormal series is also true. Moreover, it is true even in a stronger and more convenient (to prove it) form proposed by Wielandt in~\cite{Wie3}. Wielandt introduced the so-called submaximal $\fX$-subgroups, a generalization of maximal $\fX$-subgroups that allows one to use inductive arguments in passing to intersections with normal and subnormal subgroups (see Lemma~\ref{lem:inductive} in Section~\ref{sec2}). According to \cite{Wie3}, a subgroup $H$ of $G$ is called a \emph{submaximal $\fX$-subgroup} (also \emph{$\fX$-sub\-max\-simal}, $H\in\smX(G)$), if there exists an embedding of $G$ in some group $G^*$ such that $G$ is subnormal in $G^* $ and $H=G\cap H^*$ for some $H^*\in\mX(G^*)$.

Wielandt announced the assertion \cite[5.4(c)]{Wie3}: \textit{if $H,K\in\smX(G)$ and the projections of these subgroups onto sections of the subnormal series $(*)$ coincide, then $H$ and $K$ are conjugate in~$G$}. Its proof was first presented in \cite[Corollary~1]{RSV}. In the present paper, we somewhat strengthen this and hence the original assertion about maximal $\fX$-subgroups.

\begin{thm}\label{th:conj} Let $G$ be a finite group with a subnormal series $(*)$ and $H\in\smX(G)$. Then any subgroup $K$ of $G$ whose projections onto sections of the series $(*)$ coincide with those of~$H$ is conjugate to $H$ in $\langle H,K\rangle$ and, in particular, $K\in\smX(G)$.
\end{thm}

The next natural question is whether it is possible to see from the projections that an arbitrary subgroup $K$ of $G$ can be embedded into a given maximal or submaximal $\fX$-subgroup~$H$. In the case of a normal series $(*)$, Wielandt pointed out some sufficient condition under which this is possible \cite[Theorem 14.1]{Wie4}: {\it an $\fX$-subgroup~$K$ of~$G$ up to conjugacy is embedded in $H\in\mX(G)$ if $G$ has a normal series~$(*)$ and $K\leqslant N_G(H\cap G_{ i-1})G_i$ for each $i=1,\dots,n$}. Further, in \cite[Open Problems to Theorem~14.1]{Wie4}, he formulated the following general problem.\medskip

\noindent
{\bf Wielandt's Problem.} {\it Let a subgroup $K$ have the same projections onto the sections of the series~$(*)$ as some subgroup $K_*\leqslant H$, where $H\in\mX(G)$. Is it true that $K$ is conjugate to a subgroup of $H?$}\medskip

We show (see examples in Section~\ref{sec2}) that the answer to this question is negative even if the series $(*)$ is normal.\medskip

Later, see \cite[5.4(b)]{Wie3}, Wielandt announced an analogue of his assertion \cite[Theorem 14.1]{Wie4} for the case of a subnormal series: {\it if an $\fX$-subgroup $K$ of~$G$ with the subnormal series $(*)$ normalizes every projection $H^i$ of $H\in\smX(G)$, then $K$ is conjugate in $\langle H,K\rangle$ to a subgroup of $H$.} The problem with the formulation of the later assertion is that the condition $K$ {\it normalizes the projection} $H^i$ obviously needs to be clarified, since $N_G(H\cap G_ {i-1})G_i$ is no longer necessarily a subgroup. The main goal of this note is to propose an exact formulation and prove Wielandt's assertion which gives a sufficient condition for a $\fX$-subgroup $K$ to be embedded into a submaximal (in particular, maximal) $\fX$-subgroup $H$ by projections $H^i$ onto factors of the subnormal series $(*)$. To this end, we need the following definitions.

We say that two subgroups $H$ and $K$ of~$G$ {\em congruent modulo the series~$(*)$} and write
$$
H\equiv K\pmod *,
$$
if $H^i=K^i$ for each $i=1,\dots,n$. This defines an equivalence relation on the set of subgroups of~$G$. Moreover, if an automorphism $\alpha\in\operatorname{Aut}(G)$ stabilizes the series~$(*)$, i.e., $G_i^\alpha=G_i$ for each $i=1,\dots,n$, then for every $H,K\leqslant G,$
$$
H\equiv K\pmod *\quad \Rightarrow\quad H^\alpha\equiv K^\alpha\pmod *.
$$

We say that an element $x$ of~$G$ {\it normalizes a subgroup}~$H$ {\em modulo the series}~$(*)$, if it stabilizes the series~$(*)$ under conjugation and
$H\equiv H^x\pmod *.$ The subset of all such elements of $G$ is called the {\it normalizer of $H$ modulo the series~$(*)$} and denoted by $N_{\,\, G}^{(*)}(H)$. Obviously, $N_{\,\,G}^{(*)}(H)$ is a subgroup of~$G$.

\begin{thm}\label{InclusionForSubnormSeries} Let $G$ be a finite group with a subnormal series $(*)$ and $H\in\smX(G)$. If an~$\fX$-subgroup $K$ lies in $N_{\,\,G}^{(*)}(H)$, then $K$ is conjugate in $\langle H, K\rangle$ to a subgroup of~$H$.
\end{thm}

The proofs of Theorems~\ref{th:conj} and~\ref{InclusionForSubnormSeries} are based on the following property of submaximal $\fX$-subgroups: {\it if $H\in\operatorname{sm}_{\fX}(G) $, then $N_G(H)/H$ is a ${\fX'}$-group}, see Lemma~\ref{lem:WHS} below (here an {\it $\fX'$-group } is a group that does not contain nontrivial $\fX$-subgroups), and use the concept of an $\fX$-separable group which goes back to the works of Chunikhin (see, e.g.,~\cite{Chun}). According to~\cite[Definition~12.8]{Wie4}, a group is said to be {\em $\fX$-separable}, if it has a normal or, equivalently, subnormal series such that each factor is an $\fX$- or $\fX'$-group. The least possible number of $\fX$-factors in such series is called the {\it $\fX$-length} of an $\fX$-separable group $G$ and is denoted by $l_{\fX}(G)$. For example, if $\fX$ is the class of $p$-groups for some prime~$p$, then any $\fX$-separable group is $p$-solvable and its $\fX$-length is the $p$-length, as in the famous Hall and Higman paper~\cite{HH}.

Every maximal $\fX$-subgroup $H$ of an $\fX$-separable group $G$ is a {\em Hall $\fX$-subgroup} ($\fX$-Hall, $H\in\operatorname{Hall}_{\fX}(G)$) that is an $\fX$-subgroup whose index is not divisible by the primes $p$ such that $\fX$ contains a subgroup of order~$p$. Therefore, all maximal $\fX$-subgroups in such an $\fX$-separable group $G$ are conjugate \cite[Theorem~12.10]{Wie4}. In particular, every $\fX$-subgroup of $K$ lies in some conjugate of~$H$. We prove (see Proposition~\ref{ForSubnormSeries} in Section~\ref{sec2}) that under the hypothesis of Theorem~\ref{InclusionForSubnormSeries} the group $\langle H, K\rangle$ is $\fX$-separable and $H$ is a Hall $\fX$-subgroup of it. We do this using induction on the length of the series and the following assertion which is of interest on its own right.

\begin{thm}\label{thm:Wielandt's} Let $G$ be a finite group with a normal series $(*)$ of length $n$, $H\in\smX(G)$, and $\Delta=\{K\leqslant G \mid K\equiv H\pmod *\} $. Then the subgroup $W=\langle\Delta\rangle$ is $\mathfrak{X}$-separable, $l_{\mathfrak{X}}(W)\leqslant n$, and $H\in\operatorname{Hall}_ {\fX}(W)$.
\end{thm}

In fact, even a little more can be proved.

\begin{cor}\label{cor:ForNormSeries} Let $G$ be a finite group with a normal series~$(*)$ of length~$n$ and $H\in\smX(G)$. Then the subgroup $N_{\,\,G}^{(*)}(H)$ is $\fX$-separable, ${l_{\mathfrak{X}}(N_{\,\,G}^{(*)}(H))\leqslant n}$, and $H\in\operatorname{Hall}_{\mathfrak{X}}(N_{\,\,G}^{(*)}(H))$.
\end{cor}

It would be interesting to know whether the assertions of Theorem~\ref{thm:Wielandt's} and Corollary~\ref{cor:ForNormSeries} are true if the series~$(*)$ is subnormal.

\section{Proofs and examples}\label{sec2}

As mentioned, the general idea underlying the proofs of our results is that a submaximal $\fX$-subgroup can be immersed into some $\fX$-separable subgroup, where it turns out to be $\fX$-Hall.

\begin{lem}{\cite[Theorems~12.10 and~13.5]{Wie4}}\label{lem:separable}
Let $G$ be a finite $\fX$-separable group. Then
\begin{itemize}
\item[(i)] $\mX(G)=\operatorname{Hall}_{\fX}(G)$ and every two elements from this set are conjugate$;$
\item[(ii)] every subgroup of $G$ is also $\fX$-separable.
 \end{itemize}
\end{lem}

This lemma allows us to use the following argument, which is similar to the well-known Frattini argument for Sylow subgroups.

\begin{lem}[Frattini argument]\label{lem:Frattini}
Let a finite group $G$ contain a normal $\fX$-separable subgroup~$V$ and $H\in\operatorname{Hall}_{\fX}(V)$. Then $G=VN_G(H)$.
\end{lem}

\begin{proof}
If $g\in G$, then $H^g\in\operatorname{Hall}_{\fX}(V)$. By Lemma~\ref{lem:separable}(i), there exists an element $x\in V$ such that $H^g=H^x$. Hence $gx^{-1}\in N_G(H)$ and $g\in N_G(H)x\subseteq N_G(H)V=VN_G(H)$.
\end{proof}

Hall $\fX$-subgroups of an arbitrary group lend themselves well to study, thanks to the remarkable inductive property presented in the next lemma.

\begin{lem}\label{lem:Hall}
If $H\in\operatorname{Hall}_{\fX}(G)$ and a subgroup $N$ is subnormal in $G$, then $H\cap N\in\operatorname{Hall}_{\fX}(N)$.
\end{lem}

Unlike maximal $\fX$-subgroups, any $\fX$-submaximal subgroup enjoys the same property.

\begin{lem}\label{lem:inductive}
If $H\in\smX(G)$ and a subgroup $N$ is subnormal in $G$, then $H\cap N\in\smX(N)$.
\end{lem}

This property immediately follows from the definition, in contrast to another fundamental property of $\fX$-submaximal subgroups, first formulated by Wielandt in \cite[5.4(a)]{Wie3}, the proof of which, however, was first presented quite recently in \cite[Theorem~2]{RSV}.

\begin{lem}\label{lem:WHS}
If $H\in\smX(G)$, then $N_G(H)/H$ is an ${\fX'}$-group.
\end{lem}

If a group $G$ has a subnormal series~$(*)$, then in inductive reasoning it is also convenient to denote by~$(*i)$ the final segment of the series~$(*)$ starting at $G_i$, i.e.,
\begin{equation*}\label{ser_i}
 G_i\geqslant G_{i+1}\geqslant \dots\geqslant G_n=1.\eqno{(*i)}
\end{equation*}

It is easy to see that the statements below are equivalent:
\begin{itemize}
   \item[(i)] $
 H_{\phantom{i}}\equiv K_{\phantom{i}}\pmod *;$
 \item[(ii)] $H_i\equiv K_i\pmod {*}$ for each $i=1,\dots, n;$
   \item[(iii)] $ H_i\equiv K_i\pmod {*{{i}}}$ for each $i=1,\dots, n.$
\end{itemize}

Induction and Lemmas~\ref{lem:inductive} and \ref{lem:WHS} allow us to prove the following proposition which yields Theorem~\ref{th:conj}.

\begin{prop}\label{conj}
Let $G$ be a group with a subnormal series $(*)$ and $H\in\smX(G)$. Then for any subgroup $K$ of $G$ such that $K\equiv H\pmod *$, the group $J=\langle H,K\rangle$ is $\fX$-separable and $H,K\in \operatorname{Hall}_{\fX}(J)$, in particular, $H$ and $K$ are conjugate in~$J$.
\end{prop}

\begin{proof}
Induction on $n$. For $n=1$, we have $H=K$ and everything is proved.

Let $n>1$. By Lemma~\ref{lem:inductive}, we have $H_1=H\cap G_1\in\smX(G_1)$ and $ H_1\equiv K_1\pmod {*{{ 1}}}$. By the induction hypothesis, $H_1$ and $K_1$ are conjugate in $\langle H_1, K_1\rangle\leqslant J$, therefore, without loss of generality, we may assume that $H_1=K_1$. In addition, the projections $H^1$ and $K^1$ coincide, so the equalities $HG_1=KG_1=JG_1$ hold, and we may also assume that $G=HG_1=KG_1=JG_1$.

For brevity, we set $T=H_1=K_1$. Now $H,K\leqslant N_G(T)$, whence $J\leqslant N_G(T)$. Note that $N_{G_1}(T)=G_1\cap N_G(T)$ is a normal subgroup of $N_G(T)=HN_{G_1}(T)$. The factor group $N_G(T)/N_{G_1}(T)\cong H/(H\cap N_{G_1}(T))$ is an $\fX$-group. On the other hand, $T\in\smX(G_1)$ and, by Lemma~\ref{lem:WHS}, the factor $N_{G_1}(T)/T$ is an $\fX'$-group. Thus, all factors of the normal series
$$
N_G(T)\trianglerighteqslant N_{G_1}(T)\trianglerighteqslant T\trianglerighteqslant 1
$$
of $N_G(T)$ lie either in $\fX$ or~$\fX'$, so this group is $\fX$-separable. By Lemma~\ref{lem:separable}(ii), its subgroup $J$ is also $\fX$-separable. The subgroup $H$, being an $\fX$-submaximal subgroup of the $\fX$-separable group~$J$, is $\fX$-Hall in it. Since $H\equiv K\pmod *$, the same is true for the subgroup $K$. Now $H$ and $K$ are conjugate in $J$ by virtue of Lemma~\ref{lem:separable}.
\end{proof}

\textbf{Remark.} Note that the proof of Proposition~\ref{conj} is very close to the proof of Corollary~1 in~\cite{RSV}, the refinement of which is Theorem~\ref{th:conj} of this paper. The reason why we present it here in details is that it reflects our approach to proving embedding theorems below.\medskip

We now give examples showing that the answer to Wielandt's problem in Section~\ref{sec1} is negative.
\smallskip

\begin{ex}\label{ex:1}
Let $G=S_{n}$, where $n\geqslant 5$, be the symmetric permutation group on the set $\Omega=\{1,2,\dots,n\}$. Suppose that the class $\fX$ consists of all finite groups whose nonabelian composition factors have orders less than $n!/2$. It is clear that in this case every maximal $\fX$-subgroup of $G$ is simply a maximal subgroup distinct from the alternating group $A_n$, and vice versa. As such a maximal subgroup of $H$ in $G$, we take the point stabilizer of $n\in\Omega$, isomorphic to $S_{n-1}$, and consider the normal series
$$G=G_0> G_1>G_2=1,$$
where $G_1=A_{n}.$ We indicate $\fX$-subgroups $K$ and $K_* $ such that $H\equiv K\pmod *$, $K_*\leqslant H$, but $K$ is not conjugate to any subgroup from~$H$.

To this end, consider the stabilizer $H_*$ of the set $\{n-1, n\}$, which is isomorphic to $S_2\times S_{n-2}$. Denote by $K_*$ the intersection of $H$ and~$H_*$, obviously isomorphic to~$S_{n-2}$. Let also $K$ be the subgroup consisting of those elements in~$H_*$ that induce even permutations on $\Omega\setminus\{n-1, n\}$. Clearly, $K\cong A_{n-2}\times S_2$. It is easy to check that the projections of $K$ and $K_*$ onto the factors of the given series are the same. On the other hand, $K_*\leqslant H$, while $K$ is not conjugate to any subgroup of~$H$, since $K$ contains elements acting on $\Omega$ without fixed points.
\end{ex}

In the above example, the subgroups $K$ and $K_*$ are not isomorphic. It is not hard to give an example of isomorphic $\fX$-subgroups congruent modulo the given series, one is contained in the maximal $\fX$-subgroup $H$, and the other is not conjugate to any subgroup of~$H$.

\begin{ex}\label{ex:2}
In Example~\ref{ex:1}, set $n=2m$, where $m\geqslant 3$ is odd. As before, $H$ is the point stabilizer of~$n\in\Omega$. Let
$$T_*=\langle(1,2)\rangle\leqslant H\quad\text{and}\quad T=\langle(1,2)\dots(2m-1,2m)\rangle.$$
Note that $T$ acts semiregularly on $\Omega$, and hence is not conjugate to any subgroup of~$H$. Meanwhile, the projections of $T$ and $T^*$ onto sections of the same normal series coincide, since any nontrivial permutation in them is odd.
\end{ex}

Recall that the normalizer $N_{\,\,G}^{(*)}(H)$ of a subgroup of $H$ modulo series~$(*)$ is the set of elements $x\in G$ such that
\begin{itemize}
 \item[$(N1)$] stabilize the series~$(*)$, i.\,e. $G_i^x=G_i$ for all $i$, and
 \item[$(N2)$] $
 {H\equiv H^x\pmod *}.
 $
\end{itemize}

If the series~$(*)$ is normal, then the condition $(N1)$ in the definition of the normalizer is satisfied automatically.
In the general case, the set of elements stabilizing the series $(*)$ coincides with $G^{(*)}=G^{(*n)}$, where for every $i=1,\dots,n$, we refer to
$$G^{(*i)}:=\bigcap_{j=i}^{n-1}N_G(G_j)$$
as the stabilizer of the series~$(*i)$. We also define the {\em normalizer} (in $G$) {\em of $H_i$ modulo the series~$(*i)$} as the set
$$N_{\,\,\,G}^{(*i)}(H_i)=\{x\in G^{(*i)}\mid H_i^x\equiv H_i \pmod {*i}\}.$$

The following properties of these normalizers are easily verified.

 \begin{lem}\label{properties} Let $G$ be a finite group with a series~$(*)$ and $H\leqslant G$.
 \begin{itemize}
  \item[(i)]$N_{\,\,G}^{(*)}(H)$ and $N_{\,\,\,G}^{(*i)}(H_i)$ are subgroups of~$G$.
   \item[(ii)] If  $H\equiv K\pmod * $, then $N_{\,\,G}^{(*)}(H)=N_{\,\,G}^{(*)}(K);$ if $H_i\equiv K_i\pmod {*i} $, then $N_{\,\,G}^{(*i)}(H_i)=N_{\,\,G}^{(*i)}(K_i)$.
    \item[(iii)] If $\Delta=\{K\leqslant G\mid K\equiv H\pmod *\} $, then $N_{\,\,G}^{(*)}(H)$ acts by conjugation on~$\Delta;$ if $\Delta_i=\{K_i\leqslant G_i\mid K_i\equiv H_i\pmod *\}$, then $N_{\,\,\,G}^{(*i)}(H_i)$ acts by conjugation on~$\Delta_i$.
    \item[(iv)] $N_{\,\,G}^{(*)}(H)= N_{\,\,\,G}^{(*0)}(H_0)\leqslant N_{\,\,\,G}^{(*1)}(H_1)\leqslant \dots \leqslant  N_{\,\,\,G}^{(*n)}(H_{n})=G$.
    \item[(v)] If the series~$(*)$ is normal, then $H\leqslant N_{G}(H)\leqslant N_{\,\,G}^{(*)}(H)$.
 \end{itemize}
 \end{lem}

Note that in the case of a subnormal series $(*)$, the normalizer $N_{\,\,G}^{(*)}(H)$ of a subgroup $H$ modulo~$(*)$ need not include~$H$.
\smallskip

\begin{proof}[Proof of Theorem~{\rm\ref{thm:Wielandt's}}]
Let $G$ have a normal series $(*)$ of length $n$, $H\in\smX(G)$, $\Delta=\{K\leqslant G \mid K\equiv H\pmod *\}$, and $W=\langle\Delta\rangle$.
We prove Theorem~\ref{thm:Wielandt's} by induction on~$n$.

If $n=1$, then $\Delta=\{H\}$ and $W=H\in\mathfrak{X}$, i.e., everything is proven.

Let $n>1$. Consider the series
\begin{equation*}
 \label{ser1}
  G_1\geqslant \dots\geqslant G_n=1.\eqno{(*1)}
\end{equation*} of length $n-1$. Let
$$\Delta_1=\{K\cap G_1\mid K\in \Delta\}\quad\text{and}\quad\Gamma=\{L\leqslant G_1\mid L\equiv H\cap G_1\pmod{*1}\}.$$ We claim that

\begin{itemize}
  \item[(a)] $V=\langle\Gamma\rangle$ is $\mathfrak{X}$-separable, $l_{\mathfrak{X}}(V)\leqslant n-1$, and $H_1=H\cap G_1\in\operatorname{Hall}_{\mathfrak{X}}(V)$;
  \item[(b)] if $K\in\Delta$, then $K\leqslant N_G(V)$;
  \item[(c)] $\Delta_1=\Gamma$.
\end{itemize}

Since $H_1\in\smX(G_1)$ by Lemma~\ref{lem:inductive}, claim (a) holds in view of the inductive hypothesis.

Let $K\in\Delta$. Since the series $(*)$ is normal, Lemma~\ref{properties}(v) implies that $K\leqslant N_{\,\,G}^{(*)}(K) $. Further, by items~(ii) and~(iv) of the same lemma, for every $L\in\Gamma$,
  $$
   N_{\,\,G}^{(*)}(K)=N_{\,\,G}^{(*)}(H)\leqslant N_{\,\,\,G}^{( *1)}(H_1)=N_{\,\,\,G}^{(*1)}(L).
  $$
Applying Lemma~\ref{properties}(iii), we conclude that $K$ acts on~$\Gamma$, so claim~(b) holds.

It is clear that $\Delta_1\subseteq\Gamma$. Let us prove the reverse inclusion. Take $L\in\Gamma$. Since the subgroup $V$ is $\fX$-separable and any subgroup of $\Delta$ normalizes $V$ by claim~(b), the subgroup $U=HV$ is also $\fX$-separable. Consequently, the normalizer $N_{U}(L)$ is $\fX$-separable. Note that $L\in \operatorname{Hall}_{\mathfrak{X}}(V)$. Let $K\in\operatorname{Hall}_{\mathfrak{X}}(N_{U}(L))$. Then ${L=K\cap V}$. According to the Frattini argument (Lemma~\ref{lem:Frattini}) $U=VN_{U}(L)$. Now it is easy to see that $K\in \operatorname{Hall}_{\mathfrak{X}}(U)$, so $KV=U=HV$. This yields $KG_1=HG_1$, and since $K_1=K\cap G_1=L\equiv H_1\pmod {*1}$, we have $K\equiv H\pmod *$. Thus, $L=K_1\in\Delta_1$, and claim (c) is proved.

Put now $W_1=W\cap G_1$ and consider the series
 $$
 W\trianglerighteqslant  W_1 \trianglerighteqslant V\trianglerighteqslant  1.
 $$
It is normal by claim~(b). Since $HG_1=KG_1$ holds for all $K\in\Delta$, we have $WG_1=HG_1$. Hence $W/W_1\cong WG_1/G_1$ is an $\mathfrak{X}$-group. Further, the Frattini argument yields $W_1=VN_{W_1}(H_1)$, so $W_1/V\cong N_{W_1}(H_1)/N_{V}(H_1)$. Since $N_{W_1}(H_1)\leqslant N_{G_1}(H_1)$ and $H_1\leqslant N_ {V}(H_1)$, the factor group $W_1/V$ is isomorphic to a section of the group $N_{G_1}(H_1)/H_1$, which, in turn, is an $\mathfrak{X}'$-group due to Lemma~\ref{lem:WHS}. Finally, claims (a)--(c) imply that $V$ is a normal $\fX$-separable subgroup of~$W$, $l_{\fX}(V)\leqslant n-1$, and $H_1\in\operatorname{Hall}_{\fX}(V)$. Thus, $l_{\mathfrak{X}}(W)\leqslant n$ and $H\in\operatorname{Hall}_{\mathfrak{X}}(W)$, as required.
 \end{proof}

Theorem~\ref{thm:Wielandt's} plays a key role in this paper, so it is worth noting that the idea to consider the subgroup $W$ defined in this theorem came to the authors while studying Wielandt's diaries~\cite{WieDay}.

\begin{proof}[Proof of Corollary~{\rm\ref{cor:ForNormSeries}}]
Let $\Delta=\{K\leqslant G\mid K\equiv H\pmod *\} $ and $W=\langle\Delta\rangle$. As follows from items (ii), (iii) and~(v) of Lemma~\ref{properties}, the subgroup $W$ is contained and normal in $N_{\,\,G}^{( *)}(H)$. By Theorem~\ref{thm:Wielandt's}, the subgroup $W$ is $\mathfrak{X}$-separable, ${l_{\mathfrak{X}}(W)\leqslant n}$, and ${H\in\operatorname{Hall}_{\mathfrak{X}}(W)}$. Thus, it suffices to show that $N_{\,\,G}^{(*)}(H)/W$ is an $\mathfrak{X}'$-group. This follows from the Frattini argument and Lemma~\ref{lem:WHS}.
\end{proof}

Note that Theorem~\ref{InclusionForSubnormSeries} does not follow from Corollary~\ref{cor:ForNormSeries}, since by the hypothesis of this theorem, the series~$(*)$ is subnormal (and need not to be normal, as in the hypothesis of the corollary). However, using induction and Theorem~\ref{thm:Wielandt's}, we are able to prove the following proposition and then Theorem~\ref{InclusionForSubnormSeries}.

\begin{prop}\label{ForSubnormSeries}
Let $G$ be a finite group with a subnormal series $(*)$ and $K$ an $\mathfrak{X}$-subgroup of~$G$. Take $i\in\{0,1,\dots n\}$ and suppose that $H_i\in\operatorname{sm}_{\mathfrak{X}}(G_i)$ and $K\leqslant N_{\,\,\,G}^{(*i)}(H_i)$. Then there exists an element $x\in G_i\cap\langle H_i,K \rangle$ such that $K\leqslant N_G(H_i^x)$ and $H_i\equiv H_i^x\pmod{*i}$.
\end{prop}

\begin{proof} We argue by induction on $m=n-i$.

If $m=0$, then $i=n$ and $H_i=G_i=1$, so $x=1$ fits.

Let $m>0$. Then, by Lemma~\ref{lem:inductive}, $$H_{i+1}=H_i\cap G_{i+1}\in \operatorname{sm}_{\mathfrak{X}}(G_{i+1})$$ and, by Lemma~\ref{properties}(iv), $$K\leqslant N_{\,\,\,G}^{(*i)}(H_i)\leqslant N_{\,\,\,G}^{(*(i+1))}(H_{i+1}).$$ By the inductive hypothesis, there exists an element  $y\in{G_{i+1}\cap\langle H_{i+1},K \rangle}$
  such that
  $K\leqslant N_G(H_{i+1}^y)$  
  and
  $H_{i+1}^y\equiv H_{i+1}\pmod{{*(i+1)}}.$
 Since $y\in G_{i+1}$, it follows that $H_i^yG_{i+1}= H_iG_{i+1}$, so $$H_i^y\equiv H_i\pmod {*i}.$$
 By Lemma~\ref{properties}(ii), $K\leqslant N_{\,\,\,G}^{(*i)}(H_i)= N_{\,\,\,G}^{(*i)}(H_i^y)$, where $y\in G_{i+1}\cap\langle H_{i+1},K \rangle\leqslant G_i\cap\langle H_i,K \rangle.$
Therefore, without loss of generality we may replace $H_i$ with $H_i^y$ and assume that $$K\leqslant N_G(H_{i+1}).$$
Then $K$ normalizes $H_i$ modulo the normal series
\begin{equation*}
G_i\trianglerighteqslant  G_{i+1}\trianglerighteqslant 1 \eqno{(**)}
\end{equation*}
of the group $G_i$. Moreover, if
$$\Delta=\{L\leqslant G_i\mid L\equiv H_i\pmod {{**}}\},$$ then $K$ normalizes the subgroup $W=\langle\Delta\rangle\leqslant G_i$ due to Lemma~\ref{properties}(iii).

Since the series $(**)$ is normal, Theorem~\ref{thm:Wielandt's} implies that the subgroup $W$ is $\mathfrak{X}$-separable and $H_i\in\operatorname{Hall}_{\mathfrak{X}}(W)$. The subgroup $KW$ is also $\mathfrak{X}$-separable, because $K$ is an $\fX$-group. By Lemma~\ref{lem:separable}(ii), the subgroup $J=\langle H_i, K\rangle\leqslant KW$ is $\mathfrak{X}$-separable too.

Take a Hall $\fX$-subgroup $T$ of $J$ that includes~$K$. Consider the subgroup $U=T\cap W$. Since  $J\cap W$ is a normal subgroup in $J$, it follows that $U\in \operatorname{Hall}_{\mathfrak{X}}(J\cap W)$ in view of Lemma~\ref{lem:Hall}.  On the other hand, it is clear that $H_i\in \operatorname{Hall}_{\mathfrak{X}}(J\cap W)$. By Lemma~\ref{lem:separable}, there is an element $x\in J\cap W\leqslant G_i$ such that $U=H_i^x$. Since $K\leqslant T$, we have $K\leqslant N_G(U)=N_G(H_i^x)$, as required.

It remains to show that $U\equiv H_i\pmod {*i}$. Since $H_i$ is an $\fX$-Hall subgroup in both $J\cap W$ and $W$, it follows that $\operatorname{Hall}_{\mathfrak{X}}(J\cap W)\subseteq \operatorname{Hall}_{\mathfrak{X}}(W)$ and $U\in \operatorname{Hall}_{\mathfrak{X}}(W)$. Therefore $H_iG_{i+1}=WG_{i+1}=UG_{i+1}$, so $H^{(i)}=U^{(i)}$. Both $H_i$ and $K$ normalize $H_{i+1}$, whence $H_{i+1} \trianglelefteqslant J$ and $H_{i+1}\leqslant T$. The inclusions $H_{i+1}\leqslant H_i\leqslant W$ yield $H_{i+1}\leqslant T\cap W=U$, so $U\cap G_{i+1}=H_{i+1}$. Thus $U\equiv H_i\pmod{*i}$, and we are done.
\end{proof}

\begin{proof}[Proof of Theorem~{\rm\ref{InclusionForSubnormSeries}}] Let $(*)$ be  a subnormal series of $G$. Assume that $H\in\smX(G)$ and $K\leqslant N_{\,\,G}^{(*)}(H)$ for an~$\fX$-subgroup $K$ of~$G$. By Proposition~\ref{ForSubnormSeries} for $i=0$, there exists $x\in \langle H,K\rangle$ such that $K\leqslant N_G(H^x)$. Since $H^x\in \smX(G)$, Lemma~\ref{lem:WHS} implies that $K\leqslant H^x$ and $K$ is conjugate to the subgroup $K^{x^{-1}}$ of~$H$.
\end{proof}

\textbf{Remark.} In fact, the subgroup $\langle H,K\rangle$ in Theorem~\ref{InclusionForSubnormSeries} is $\mathfrak{X}$-separable. Indeed, as in the proof of this theorem, there exists $x\in\langle H,K\rangle$ such that $K\leqslant H^x$ and $H^x\equiv H\pmod *$. Now Proposition~\ref{conj} implies that the group $\langle H,H^x\rangle$ and its subgroup $\langle H,K\rangle$ are $\mathfrak{X}$-separable.

\end{document}